\documentclass[12pt]{amsart}
\usepackage{geometry}     
\usepackage{graphicx}
\usepackage{color}

\usepackage{url}
\usepackage{amsthm}
\usepackage{amsmath}
\usepackage{amsfonts}
\usepackage{enumerate}
\usepackage{amssymb}
\usepackage{epstopdf}
\usepackage{hyperref}
\usepackage{mathrsfs}

\newtheorem{theorem}{Theorem}[section]
\newtheorem{proposition}[theorem]{Proposition}
\newtheorem{lemma}[theorem]{Lemma}

\newtheorem{conjecture}[theorem]{Conjecture}

\DeclareMathOperator{\Vol}{Vol}
\DeclareMathOperator{\interior}{int}
\DeclareMathOperator{\Hess}{Hess}
\DeclareMathOperator{\length}{L}
\DeclareMathOperator{\sym}{Sym}
\DeclareMathOperator{\per}{Per}
\DeclareMathOperator{\dl}{dL}
\DeclareMathOperator{\dv}{dVol}
\newcommand{\R}{\mathbb{R}}

\newcommand{\Z}{\mathbb{Z}}

\newcommand{\M}{\textbf{M}}

\newcommand{\eps}{\varepsilon}

\theoremstyle{definition}
\newtheorem{definition}[theorem]{Definition}

\title{Generic density of geodesic nets}
\author{Yevgeny Liokumovich and Bruno Staffa}

\begin{document}

\maketitle

\begin{abstract}
We prove that for a Baire-generic Riemannian metric
on a closed smooth manifold, the union of the images of all
stationary geodesic nets
forms a dense set.
\end{abstract}

\section{Introduction}
A weighted multraph is a finite one-dimensional simplicial complex $\Gamma$ with a multiplicity $n(E)\in\mathbb{N}$ assigned to each edge
(1-dimensional face) $E$ of $\Gamma$. A geodesic net is a map from a weighted multigraph $\Gamma$ 
to a Riemannian manifold $(M,g)$, whose edges are
geodesic segments in $M$. 
A geodesic net is called stationary
if it is a critical point of the length functional $\length_{g}$ with respect to $g$.
This is equivalent to the condition that the sum of the inward pointing
unit tangent vectors (with multiplicity) is zero at every vertex
(see \cite{NP} for background on stationary geodesic nets and 
open problems).

In this paper we prove the following result.

\begin{theorem} \label{density}
Let $M^n$, $n \geq 2$,
be a closed manifold
and let $\mathcal{M}^k$ be the space of $C^k$
Riemannian metrics on $M$, $3\leq k \leq \infty$. 
For a generic (in the Baire sense) subset of $\mathcal{M}^k$ the union of the images of all 
embedded stationary geodesic nets in $(M,g)$ is dense.
\end{theorem}

An analogous density result for closed geodesics on surfaces was
proved by Irie \cite{I}. For minimal hypersurfaces in Riemannian
manifolds of dimension $3 \leq n \leq 7$ a generic density result was
proved by Irie-Marques-Neves \cite{IMN}. 

\vspace{0.1in}

\textbf{Acknowledgements.} We are grateful to Otis Chodosh, Alexander Kupers and
Christos Mantoulidis for their valuable comments and suggestions.
We would like to thank the anonymous referee for their valuable feedback.
The authors were partially supported by NSERC Discovery grant. Y.L. was partially supported by Sloan Fellowship.

\section{$\Gamma$-nets}
Fix a weighted multigraph $\Gamma$ and a closed manifold $M$.

\begin{definition}
A $\Gamma$-net on $M$ is a continuous map $f: \Gamma \rightarrow M$ which is a $C^{2}$ immersion when restricted to each edge.
\end{definition}

\begin{definition}
Given a Riemannian metric $g$ on $M$, we say that a $\Gamma$-net $f$ is stationary with respect to $g$ if it is a critical point of the length functional $\length_{g}$. The previous holds if for every one parameter family $\Tilde{f}:(-\delta,\delta)\times\Gamma\to M$ of $\Gamma$-nets with $\Tilde{f}(0,\cdot)=f$ we have
\begin{equation*}
    \frac{d}{ds}\bigg|_{s=0}\length_{g}(\Tilde{f}_{s})=0
\end{equation*}
where $\Tilde{f}_{s}=\Tilde{f}(s,\cdot)$. A more detailed discussion can be found in \cite[Section~1]{St}.
\end{definition}

\begin{definition}
 We say that a $\Gamma$-net is embedded if $f:\Gamma\to M$ is injective (notice that by compactness of $\Gamma$, this implies that $f$ is a homeomorphism onto its image). We denote by $\Omega(\Gamma,M)$ the space of embedded $\Gamma$-nets on $M$.
\end{definition}

\begin{definition}
A weighted multigraph is good* if it is connected and each vertex $v \in \mathscr{V}$ has at least three different incoming edges. A weighted multigraph is good if either it is good* or it is a simple loop with multiplicity.
\end{definition}

Given a stationary geodesic net $f$, we can always find an embedded stationary geodesic net $\Tilde{f}$ with the same image and multiplicity as $f$ at every point.

\begin{lemma} \label{regular}
Let $f: \Gamma \rightarrow (M,g)$ be a stationary geodesic net. Then there exist an embedded stationary geodesic net $\Tilde{f}:\Tilde{\Gamma}\to (M,g)$ which has the same image with multiplicity as $f$ and the property that each connected component of $\Tilde{\Gamma}$ is good. In particular, it holds $\length_{g}(f) = \length_{g}(\Tilde{f})$.
\end{lemma}

\begin{proof}
First of all, we can find an injective stationary geodesic net $f':\Gamma'\to M$ which has the same image with multiplicity as $f$. This can be done as follows.
\begin{enumerate}
    \item Firstly, we replace the weighted multigraph $\Gamma$ by a new one such that for every edge $E$, the map $f|_{E}$ does not have any self-intersections. This is done by subdividing each edge $E$ in equal parts $E_{1},...,E_{l}$ so that the length of $f(E_{i})$ is not bigger than the injectivity radius of $(M,g)$ for every $1\leq i\leq l$. 
    \item Once the previous is done, suppose we have two different edges $E_{1}$ and $E_{2}$ with multiplicities $n_{1}$ and $n_{2}$ respectively whose interiors overlap non-transversally. 
    Assume $f(E_1) \cap f(E_2)$ is connected and that their symmetric difference is non-empty. The cases when $f(E_1) \cap f(E_2)$ has
    two components or  $f(E_i) \subset f(E_j)$ are treated similarly.
    
    Let $v_{11},v_{12}$ be the vertices of $E_{1}$ and $v_{21},v_{22}$ be the vertices of $E_{2}$. Then we can remove $E_{1}$ and $E_{2}$, and replace them by three new edges: $E_{3}$ which has vertices $v_{11}$ and $v_{21}$, multiplicity $n_{1}$ and represents the part of $E_{1}$ where there is no overlap with $E_{2}$; $E_{4}$ which has vertices $v_{21}$ and $v_{12}$, multiplicity $n_{1}+n_{2}$ and represents the overlap between $E_{1}$ and $E_{2}$; and $E_{5}$ which has vertices $v_{12}$ and $v_{22}$, multiplicity $n_{2}$ and represents the part of $E_{2}$ where there is no overlap with $E_{1}$. Observe that after applying this procedure, the edges of the new graph are still mapped to geodesic segments of length bounded by the injectivity radius of $(M,g)$, and therefore such curves do not have any self intersections. As each time we do this operation the number of pairs of edges whose interiors intersect non-transversally at some point decreases, eventually we will get a new weighted multigraph such that if two edges intersect at an interior point, then the intersection is transverse.
    \item After the previous step, if $f(E_{1})$ intersects $f(E_{2})$, then $E_{1}\neq E_{2}$ and the intersection is transverse. Consider an intersection point $P$ between $f(E_{1})$ and $f(E_{2})$, $E_{1}\neq E_{2}$ edges. Let $v_{11},v_{12}$ and $v_{21},v_{22}$ be the vertices of $E_{1}$ and $E_{2}$ respectively. We can introduce a new vertex $v$ which will be mapped to $P$ and replace $E_{1},E_{2}$ by $E_{3},E_{4},E_{5},E_{6}$ where $E_{3},E_{4}$ are obtained by the subdivision of $E_{1}$ induced by $P$, and $E_{5},E_{6}$ are obtained by the subdivision of $E_{2}$ induced by $P$. After doing this operation with each intersection point $P$ of the images of different edges, we will obtain a geodesic net $f:\Gamma\to M$ such that given any two different edges $E_{1},E_{2}$, $f(E_{1})$ and $f(E_{2})$ do not overlap at any interior point and no edge self-intersects.
    \item At this point, if $f(t_{1})=f(t_{2})$ for some $t_{1}\neq t_{2}$, then both $t_{1}$ and $t_{2}$ must be vertices. Denote $v_{j}=t_{j}$ for $j=1,2$. If we replace $\Gamma$ by the quotient graph obtained by identifying $v_{1}$ and $v_{2}$, and iterate this procedure each time it is possible, we obtain an injective stationary geodesic net $f:\Gamma\to M$.
\end{enumerate}

Now we perform some changes to ensure that each connected component of $\Gamma'$ is good. We do this component by component, so we can assume that we start from an embedded stationary geodesic net $f':\Gamma'\to M$ where $\Gamma'$ is connected. In such situation, consider a vertex $v$, such that all edges adjacent to $v$ have colinear tangent vectors at $v$. We assume that $\Gamma'$ is not a simple loop with multiplicity, as in that case we are done. Since the vertex is balanced, there exist edges $E_1$ with multiplicity $n_1$ (with vertices $v_1$ and $v$) and $E_2$ with multiplicity $n_2$ (with vertices $v$ and $v_2$) with opposite inward tangent vectors at $v$ and $v_{i}\neq v$ for $i=1,2$. As the map $f'$ is injective, it must be $n_{1}=n_{2}$ and $E_{1},E_{2}$ should be the only edges at $v$ (if not, there would be another edge $E_{3}$ concurring at $v$ with the same inward tangent vector as $E_{i}$ for some $i\in\{1,2\}$, and as $E_{3},E_{i}$ are mapped to geodesics, their images would coincide along an interval). Thus if $v_{1}\neq v_{2}$, we can define a new graph $\Gamma'$ by deleting $v$, $E_{1}$ and $E_{2}$, and adding an edge $E$ connecting $v_{1}$ and $v_{2}$ with multiplicity $n_{1}=n_{2}$ and image $f'(E_{1})\cup f'(E_{2})$. This operation keeps $\Gamma'$ connected and $f'$ injective. If $v_{1}=v_{2}$, the previous construction gives us a simple geodesic loop with multiplicity $n_{1}=n_{2}$. Iterating this construction, we eventually obtain a new $\Tilde{f}:\Tilde{\Gamma}\to M$ such that $\Tilde{\Gamma}$ is either a simple loop with multiplicity or it satisfies that each of its vertices $v$ admits two incoming edges $E_{1},E_{2}$ such that $\Tilde{f}(E_{1})$ and $\Tilde{f}(E_{2})$ have different tangent lines at $\Tilde{f}(v)$. In the latter case, the condition that the sum of the unit inward tangent vectors at $v$ should be $0$ forces there to be at least three different incoming edges at $v$ making $\Tilde{\Gamma}$ a good* weighted multigraph. This completes the proof.
\end{proof}

Following \cite{St} we say that a stationary geodesic net $f$ is non-degenerate if every null vector of $\Hess_{f}\length_{g}$ is parallel along $f$. The following result is a consequence of the Implicit Function Theorem and is proved for embedded $\Gamma$-nets when $\Gamma$ is good* in \cite[Lemma~4.6]{St}. The same argument can be adapted to closed geodesics using the Structure Theorem of Brian White proved in \cite{Wh}. A more elementary proof can be obtained considering the finite dimensional models of the spaces of geodesic nets (instead of working with the infinite dimensional $\Omega(\Gamma,M)$ as in \cite{St}). 

\begin{lemma} \label{nondegenerate}
 Let $\Gamma$ be a good weighted multigraph and  $f_{0}: \Gamma \rightarrow M$ be an embedded non-degenerate stationary geodesic net with respect to a $C^{k}$ metric $g_{0}$, $k\geq 3$. Then there exists a neighborhood $W$ of $g_{0}$ in $\mathcal{M}^{k}$ and a differentiable map $\Delta:W\to\Omega(\Gamma,M)$ such that $\Delta(g)$ is a non-degenerate stationary geodesic net with respect to $g$ for every $g\in W$.
\end{lemma}

Let $\mathcal{S}^k(\Gamma)$ denote the set of pairs $(g, [f])$,
where $g \in \mathcal{M}^k$ and $[f]$ denotes the equivalence class (up to reparametrization) of an embedded stationary
$\Gamma$-net $f$ with respect to $g$, as defined in \cite{St}.  The following structure theorem for the space of embedded stationary geodesic nets, analogous
to White's structure theorem for minimal submanifolds
\cite{Wh}, was proved by 
Staffa in \cite{St}
(a similar structure theorem for
stationary geodesic nets on surfaces was independently obtained
by Chodosh and Mantoulidis in \cite{CM}).

\begin{theorem} \label{structure}
Let $\Gamma$ be a good weighted multigraph. The space $\mathcal{S}^k(\Gamma)$ is a second countable $C^{k-2}$ Banach manifold
and the projection map $\Pi: \mathcal{S}^k(\Gamma) \rightarrow \mathcal{M}^k$
is a $C^{k-2}$ Fredholm map of Fredholm index $0$. 
For a regular value $g \in \mathcal{M}^k$ the set $\Pi^{-1}(g)$
is a countable collection of non-degenerate embedded stationary geodesic nets.
\end{theorem}

By Sard-Smale theorem \cite{Sm} a generic metric $g \in \mathcal{M}^k$
is a regular value of $\Pi$.

\section{Min-max constructions}

Stationary geodesic nets arise from Almgren-Pitts
Morse theory on the space of 1-cycles.

By Almgren isomorphism theorem (\cite{A1}, \cite{A2}, \cite{GuL})
the space of mod 2 $k$-cycles on the $n$-sphere $\mathcal{Z}_k(S^n,\Z_2)$
is weakly homotopy equivalent to the Eilenberg-MacLane
space $K(\Z_2, n-k)$. Let $\overline{\lambda}$ denote the
non-trivial element of $H^{n-k}(\mathcal{Z}_k(S^n,\Z_2);\Z_2)$.
Note that all cup powers of $\overline{\lambda}$ are non-trivial and
the cohomology ring of $\mathcal{Z}_k(S^n,\Z_2)$
is generated by the cup powers and Steenrod squares of $\overline{\lambda}$
(\cite{Ha}).

Given a closed $n$-dimensional Riemannian manifold $(M,g)$ consider $\phi: M \rightarrow S^n$ that maps a small open ball $B \subset M$
diffeomorphically onto $S^n \setminus \{p\}$ and sends the rest of $M$ to point $\{p\}$.
For the corresponding map on the space of cycles $\Phi: \mathcal{Z}_k(M,\Z_2)
\rightarrow \mathcal{Z}_k(S^n,\Z_2)$ the pull-back 
$\lambda = \Phi^*(\overline{\lambda}) \neq 0$.

Given a simplicial complex $X$ we say that $F:X \rightarrow \mathcal{Z}_k(M,\Z_2)$
is a $p$-sweepout if $F^*(\lambda^p) \neq 0 \in H^{p(n-k)}(X;\Z_2)$
and $F$ satisfies a no-concentration of mass property (cf. \cite{MN1}, \cite{LMN}).
We define the $k$-dimensional $p$-width $\omega_p^k(M,g)$ by
$$\omega_p^k(M,g) = \inf \{\sup_{x\in X} \M(F(x)): F\text{ is a }p\text{-sweepout of }M\}$$

Using arguments of \cite{Gro1}, \cite[Section 8]{Gro2}, \cite{Gu2} 
we obtain the following upper bounds for the widths $\omega_p^k(M,g)$.

\begin{proposition} \label{upper bound}
Let $(M,g)$ be a closed $n$-dimensional Riemannian manifold. There exists a constant $C=C(g)$,
such that $\omega_p^k(M,g) \leq C p^{\frac{n-k}{n}}$.
\end{proposition}

\begin{proof}
The case of $k=n-1$ was proved in \cite[Theorem 5.1]{MNRicci}.
Assume $1 \leq  k \leq n-2$. Let $\sym_p S^{n-k}$ denote the symmetric product of spheres
$\sym_p S^{n-k}= \{(x_1,...,x_p): x_i \in S^{n-k}\}/\per(p)$, where $\per(p)$
is the group of permutations of $p$ elements. For $1 \leq j \leq p$ we have that
$H^{j(n-k)}(\sym_p S^{n-k})= \langle\alpha^j\rangle$, where $\alpha$ is the non-trivial
cohomology class in $H^{n-k}(\sym_p S^{n-k})$ (we are considering cohomology with $\mathbb{Z}_{2}$ coefficients, see \cite{Na}).
In \cite{Gu2} Guth constructed $p$-sweepouts $F_p:\sym_p S^{n-k} \rightarrow 
\mathcal{Z}_{k}(B, \partial B;\Z_2)$ of the Euclidean unit ball
$B \subset \R^n$ by piecewise linear relative $k$-cycles satisfying
$$\sup\{ \M(F_p(x)): x \in \sym_p S^{n-k}\} \leq C_n p^{\frac{n-k}{n}}$$

Fix a fine triangulation and PL structure on $M$ that is bilipschitz equivalent
to the original metric $g$, and let $\Phi: M \rightarrow \R^{n}$
be a PL map, such that each simplex $\Delta$ is bilipschitz to $\Phi(\Delta)$.
After scaling we may assume that $\Phi(M) \subset \interior (B)$. If $z$ is a piecewise linear
relative cycle in $B$, then $\Phi^{-1}(z)$ is a $k$-cycle in $M$.
The map $F_p':\sym_p S^{n-k} \rightarrow \mathcal{Z}_{k}(M;\Z_2)$
defined as
$F_p'(x) = \Phi^{-1} (F_p(x))$ satisfies the desired mass bound.
To see that this is a $p$-sweepout consider the restriction 
of $F_p'$ to $\{[x,0,...,0]: x \in S^{n-k}\}\subset \sym_p S^{n-k}$.
It is straightforward to check that Almgren gluing map (\cite{A1}) maps this family to the fundamental homology class of $M$, so 
$(F_p')^*(\lambda)= \alpha \in H^{(n-k)}(\sym_p S^{n-k})$.

\end{proof}

Almgren showed that widths correspond to volumes of stationary integral
varifolds. For 1-dimensional widths a stronger regularity result
is known (see \cite{A1}, \cite{A2}, \cite{CaCa}, \cite{NR}, \cite{Pi1}, \cite{Pi2}), namely, that the stationary integral 1-varifolds are, in fact, stationary geodesic nets.
Combining this result with Lemma \ref{regular} we obtain the following.

\begin{proposition} \label{width}
The width $\omega_p^1(M,g) = \sum_{i=1}^{P} \length_{g}(\gamma_i)$, where $\gamma_i:\Gamma_{i}\to M$ is an embedded stationary geodesic net and $\Gamma_{i}$ is a good weighted multigraph for each $1\leq i\leq P$.
\end{proposition}

In \cite{IMN} density of minimal hypersurfaces was proved
using a Weyl law for $(n-1)$-dimensional $p$-widths. The Weyl law was proved
for $(n-1)$-cycles in arbitrary compact manifolds and for $k$-cycles in Euclidean domains  in \cite{LMN}.
However, it is not known in general for $k<n-1$, although
the special case of 1-cycles in 3-manifolds has been resolved recently \cite{GuL}.

In \cite{So2} Song
 observed that the full strength
of the Weyl law is not needed to prove
density of minimal hypersurfaces for generic metrics.
(It does, however, seem that the Weyl law
is necessary to prove a stronger 
equidistribution result in \cite{MNS}). The idea of Song
allows us to circumvent the use of Weyl law to prove
density of stationary geodesic nets.

\begin{lemma} \label{widthincrease}
Let $g_1$ and $g_2$ be two metrics
on $M$ with $g_2 \geq g_1$ and $g_2(x_{0}) > g_1(x_{0})$ for some $x_{0} \in M$. Then there exists $p \geq 1$, such that $\omega^k_p(M,g_2) > \omega^k_p(M,g_1)$.
\end{lemma}

\begin{proof}
Let $B_r(x_{0})$ be a small closed ball
such that $g_2>g_1$ on $B_r(x_{0})$.
Fix $\eps>0$, such that
for every $k$-cycle $z$ with $g_2$-mass
$\M_{g_2}(z \llcorner  B_r(x_{0})) > \frac{1}{2} \omega_1^k(B_r(x_{0}),g_2)$
we have $\M_{g_2}(z) - \M_{g_1}(z)> \eps$.

By Proposition \ref{upper bound} we have
$\omega_p^k(M, g_1) \leq C p^{\frac{n-k}{n}}$ for some constant $C>0$.
In particular, we can find $p>0$, such that
$\omega_p^k(M,g_1) - \omega_{p-1}^k(M,g_1) < \eps/4$.
Let $F: X \rightarrow Z_k(M; \Z_2)$ be a
$p$-sweepout of $(M,g_{2})$ such that 
$\M_{g_2}(F(x)) \leq \omega^k_p(M,g_2) + \eps/4$ for all $x\in X$.
By \cite[Lemma 2.15]{LMN} we can assume that 
the map $F$ is continuous in the mass norm.

Recall that if two manifolds are bilipschitz diffeomorphic,
then the corresponding spaces of cycles are homeomorphic. In particular, a $p$-sweepout
of one induces a $p$-sweepout of the other. Let $X_1=\{x \in X: \M_{g_2}(F(x)\llcorner B_{r}(x_{0})) > \frac{1}{2} \omega_1^k(B_{r}(x_{0}),g_2)\} $ be an open subset of  $X$.
We claim that the restriction of
$F$ to $X_1$ is a $(p-1)$-sweepout of $M$ (with respect to both $g_{1}$ and $g_{2}$ as $(M,g_{1})$ and $(M,g_{2})$ are bilipschitz diffeomorphic).
Indeed, let
$\lambda\in H^{n-k}(\mathcal{Z}_{k}(M,\mathbb{Z}_{2}))$ be the non-trivial class defined before.
Then $\lambda$ vanishes on
 $X \setminus X_1$ because $F|_{X\setminus X_{1}}\llcorner B_{r}(x_{0})$ is not a sweepout of $B_{r}(x_{0})$ and hence $F|_{X\setminus X_{1}}$ can not be a sweepout of $M$. If $\lambda^{p-1}$
 vanishes on $X_1$, then $\lambda^p$ vanishes
 on $X_1 \cup (X \setminus X_1) =X$,
 which contradicts the definition of
 $p$-sweepout.
 
It follows that $\{ F(x) \}_{x \in X_1}$ is a $(p-1)$-sweepout of $M$ and
\begin{equation*}
    \begin{split}
    \omega_{p-1}^k(M, g_1) & \leq \sup\{ \M_{g_1}(F(x)): x \in X_1\} \\
    & \leq \sup\{\M_{g_2}(F(x)): x \in X_1\} - \eps \\
    & \leq \omega_p^k(M, g_2) - 3/4\eps
    \end{split}
\end{equation*}
If $\omega_p^k(M, g_2) = \omega_p^k(M, g_1)$ then our choice
of $p$ leads to a contradiction.
\end{proof}

The next Lemma follows as in \cite[Lemma 1]{MNS}.

\begin{lemma} \label{lipschitz}
Let $M$ be a closed manifold. Then the k-dimensional p-width $\omega^k_p(g)$ is a locally Lipschitz function
of the metric $g$ in the space $\mathcal{M}^{0}$ of $C^0$ metrics.
\end{lemma}

\begin{proof}
First we need to give a metric space structure to the set $\mathcal{M}^{0}$. Observe that each $g\in\mathcal{M}^{0}$ induces a metric $d_{g}$ in $\mathcal{M}^{0}$ defined as

\begin{equation*}
    d_{g}(g_{1},g_{2})=\sup_{v\neq 0}\frac{|g_{1}(v,v)-g_{2}(v,v)|}{g(v,v)}
\end{equation*}

It is easy to show that as $M$ is compact, given $g,g'\in\mathcal{M}^{0}$ the induced metrics $d_{g}$ and $d_{g'}$ are equivalent. Therefore we can pick an arbitrary $g_{0}\in\mathcal{M}^{0}$ and fix $d_{g_{0}}$ as our metric.

Now in order to prove the lemma, fix a metric $g\in\mathcal{M}^0$ and suppose $g_1,g_2$ satisfy $ g/C_1 \leq g_i \leq C_1g$ for $i=1,2$ and some $C_1>1$. For some constant $C=C(g)>0$ we have
$\omega_p^k(M, g) \leq C p^{\frac{n-k}{n}}$ by Proposition \ref{upper bound}.

Given a $k$-cycle $z \in \mathcal{Z}_k(M;\Z_2)$ we have
\begin{equation*}
    \begin{split}
    \M_{g_1}(z) - \M_{g_2}(z) & \leq \Big(\big(\sup_{v \neq 0} \frac{g_1(v,v)}{g_2(v,v)}\big)^{\frac{k}{2}}-1\Big)\M_{g_2}(z) \\
    & \leq \Big(\big(1+ \sup_{v \neq 0}\frac{|g_1(v,v) - g_2(v,v)|}{g_2(v,v)}\big)^{\frac{k}{2}}-1\Big)\M_{g_2}(z) \\
    &\leq \Big(\big(1+C_1 d_{g}(g_1,g_2))^{\frac{k}{2}}-1\Big)\M_{g_2}(z)\\
    & \leq C_1k d_{g}(g_1,g_2)\M_{g_2}(z)
    \end{split}
\end{equation*}
for small $d_{g}(g_1,g_2)$.

Then for $g_1$, $g_2$ near $g$ we have
\begin{equation*}
    \begin{split}
    |\omega_{p}^k(M, g_1)- \omega_{p}^k(M, g_2)|
    & \leq C_1kd_{g}(g_1,g_2) \omega_{p}^k(M, g_2) \\
    & \leq C_1^{1+\frac{k}{2}}k C  p^{\frac{n-k}{n}}d_{g}(g_1,g_2)
    \end{split}
\end{equation*}

As $d_{g}$ is equivalent to $d_{g_{0}}$ we get the desired result.
\end{proof}

\section{Proof of the main theorem}
Fix a manifold $M$ and an open subset $U \subset M$.
Let $\mathcal{M}^{k}(U) \subset \mathcal{M}^k$ denote the set of $C^{k}$ metrics $g$
such that there exists an embedded non-degenerate stationary geodesic net in $(M,g)$ intersecting $U$ whose domain is a good weighted multigraph. First we will analyse the case $3\leq k<\infty$.

By Lemma \ref{nondegenerate} we have that $\mathcal{M}^k(U)$ is open.
Now we will show that $\mathcal{M}^k(U)$ is dense. Let $V\subseteq\mathcal{M}^{k}$ be an open subset. We have to show that there exists some $g\in V\cap\mathcal{M}^{k}(U)$.

Let $\{\Gamma_{m}\}_{m\in\mathbb{N}}$ be the countable collection of all good weighted multigraphs. Let $\mathcal{C}_m =  \mathcal{S}^k(\Gamma_m)$.
We have that the projection map $\Pi_{m}: \mathcal{C}_m \rightarrow \mathcal{M}^k$
is a Fredholm map of index $0$ by Theorem \ref{structure}.
Let $Reg_m \subset \mathcal{M}^k$ denote the set of regular values
of $\Pi_m$ and $R = \bigcap_{m\geq 0} Reg_m $.
By Sard-Smale theorem the set $R$ is comeager, so we can find a metric 
$g_0 \in V \cap R$. If $g_{0}\in\mathcal{M}^{k}(U)$ we are done, so let us assume the contrary. Then all embedded stationary geodesic nets of $(M,g_0)$ with domain a good weighted multigraph
are non-degenerate and do not intersect $U$. Let $L_0$ denote the (countable) set of lengths of such geodesic networks. By Lemma \ref{regular}, the set $L_{1}$ of lengths of all stationary geodesic nets for the metric $g_{0}$ is the set of finite sums of elements in $L_{0}$, and hence it is also countable.

Let $\phi: M \rightarrow \R$
be a non-negative smooth bump function supported in $U$ with $\phi(x_{0})>0$
for some $x_{0} \in U$. Define $g_t(x) = (1+t\phi(x))g_0(x)$.
For some sufficiently small $\eps>0$ we have that
$g_t \in V$ for all $t \in [0,\eps]$.
By Lemma \ref{widthincrease} there exists $p>0$, such that
$\omega^1_p(g_\eps)>\omega^1_p(g_0)$.

By Smale's transversality theorem from \cite{Sm}, there exists a sequence of embeddings 
$g_i:[0,\eps] \rightarrow \mathcal{M}^k$ converging to $g$, such that each $g_i$ 
is transverse
to the maps $\Pi_m:\mathcal{C}_m \rightarrow \mathcal{M}^k$ for all $m\geq0$.
Moreover, using \cite[Theorem~3.3]{Sm} we have that $I_{i,m}=  \Pi_m^{-1}(g_i([0,\eps]))$
is a $1$-dimensional submanifold of $\mathcal{C}_{m}$ for each $(i,m)\in\mathbb{N}\times\mathbb{N}_{0}$.
Notice that by transversality,  if $t$ is a regular value of $(g_i)^{-1}\circ \Pi_m|_{I_{i,m}}$, then
$g_i(t) \in Reg_m$ (cf. \cite[Lemma 2]{MNS}).
By the finite-dimensional Sard's lemma applied to $(g_i)^{-1}\circ \Pi_m|_{I_{i,m}}$ we have that $C_i =\bigcap_{m\geq 0}\{t:g_i(t) \in Reg_m \}\subset [0, \eps]$ 
is a subset of full measure.


Note that $\omega_p^1(g_{i}([0,\eps])) \rightarrow \omega^1_p(g([0,\varepsilon]))$ as $i \rightarrow \infty$ 
and without any loss of generality we may assume that there is an interval
 $[a,b] \subset \omega_p^1(g_{i}([0,\eps])) $ for all $i$.
By Lemma \ref{lipschitz} we have that $C= \bigcap_{i=1}^\infty \omega_p^1(g_{i} (C_i)) \cap [a,b] \setminus L_{1}$ is non-empty (because $L_{1}$ is countable and $\omega^{1}_{p}(g_{i}(C_{i}))\cap[a,b]$ is a full measure subset of $[a,b]$ for every $i\in\mathbb{N}$). Let $l \in C$.
By Proposition \ref{width},   
for each $i$ we have that
$l = \sum_{j=1}^{P_i} L(\gamma_j^i)$, where each $\gamma^i_j$ is an
(non-degenerate) embedded stationary geodesic net in $(M,g_i(t_i))$ whose domain is a good weighted multigraph,
for some $t_i \in (\omega_p^1 \circ g_i)^{-1}(l)$.
Passing to a subsequence if necessary, we can assume that there exists $t'=\lim t_i \in [0,\varepsilon]$ and that the sequence $\gamma^i = \bigcup_j \gamma_j^i$
converges to a stationary geodesic net $\gamma$ in $(M,g_{t'})$. However, since $L(\gamma)=l \notin L_1$, $\gamma$ is not a stationary geodesic net for $g_0$ and hence it must intersect $U$. As $\lim\gamma^{i}=\gamma$, there exists $i_{1}\in\mathbb{N}$ such that $\gamma^{i}$ intersects $U$ for all $i\geq i_{1}$. On the other hand, as $\lim g_{i}(t_{i})=g_{t'}\in V$, there exists $i_{2}\in\mathbb{N}$ such that $i\geq i_{2}$ implies $g_{i}(t_{i})\in V$. Thus if $i\geq\max\{i_{1},i_{2}\}$, the metric $g_{i}(t_{i})$ is in $V$ and one component $\gamma^{i}_{j}$ of $\gamma^{i}$ is an embedded stationary geodesic net intersecting $U$ whose domain is a good weighted multigraph. As $g_{i}(t_{i})$ is bumpy, we deduce that $g_{i}(t_{i})\in\mathcal{M}^{k}(U)$ and hence $V\cap\mathcal{M}^{k}(U)\neq\emptyset$.


So far we have proved that for $3\leq k<\infty$, $\mathcal{M}^{k}(U)\subseteq\mathcal{M}^{k}$ is open and dense for every open subset $U\subseteq M$. Taking a countable basis $\{U_{m}\}_{m\in\mathbb{N}}$ for the topology of $M$ and setting $\mathcal{N}^{k}=\bigcap_{m\in\mathbb{N}}\mathcal{M}^{k}(U_{m})$ we see that $\mathcal{N}^{k}\subseteq\mathcal{M}^{k}$ is generic and $g\in\mathcal{N}^{k}$ if and only if the union of the images of all nondegenerate embedded stationary geodesic networks with respect to $g$ whose domain is a good weighted multigraph is dense in $M$. This proves Theorem \ref{density} in the case $3\leq k<\infty$. For the case $k=\infty$, we can define $\mathcal{N}^{\infty}$ to be the set of $C^{\infty}$ metrics for which the union of the images of all nondegenerate embedded stationary geodesic nets whose domain is a good weighted multigraph is dense in $M$. Thus it is clear that $\mathcal{N}^{\infty}=\bigcap_{k\geq 3}\mathcal{N}^{k}$ and that if $k'\geq k$ then $\mathcal{N}^{k'}=\mathcal{N}^{k}\cap\mathcal{M}^{k'}$; so by \cite[Lemma~6.2]{St} we deduce that $\mathcal{N}^{\infty}$ is a generic subset of $\mathcal{M}^{\infty}$ (see also a similar 
argument in \cite{Wh2} and \cite[Corollary 5.14]{CM}).

\section{Open problems}
By analogy with the case of minimal hypersurfaces \cite{MNS},
we conjecture that an equidistribution result should hold
for stationary geodesic nets. 

\begin{conjecture} \label{equidistribution}
For a generic set of metrics, there exists a set of stationary geodesic 
nets that is equidistributed in $M$. Specifically, for every $g$ in the generic set, 
there exists
a sequence $\{ \gamma_i: \Gamma_i \rightarrow M \}$ of stationary geodesic nets in $(M,g)$,
such that for every $C^{\infty}$ function $f: M \rightarrow \R$ we have
$$\lim_{k \rightarrow \infty} \frac{\sum_{i=1}^k \int_{\gamma_i} f \dl_g}{\sum_{i=1}^k \length_{g}(\gamma_i)}= \frac{\int_M f \dv_g}{\Vol(M,g)}$$
\end{conjecture}

The cases $n=2$ and $n=3$ of this conjecture were solved in \cite{LiSta}. In fact, in dimension $n=2$
it is proved that closed geodesics are equidistributed in $M$ for generic metrics.

By analogy with Yau's conjecture for minimal surfaces recently resolved by Song \cite{So1} we also conjecture that there exist infinitely many distinct stationary geodesic nets in every Riemannian manifold $(M,g)$.

\bibliographystyle{amsbook}

\end{document}